\documentclass[a4paper,12pt]{article}     % `article' for A4 paper
\usepackage{amsmath,amscd,amssymb}        % AMS-LaTeX with CD and symbols
\usepackage{latexsym}                     % old LaTeX symbols
\usepackage[english, german]{babel}                % German & English (default)
\usepackage{xypic}

\usepackage{theorem}                 % theorem extensions
\usepackage{color}

\newtheorem{theorem}{Theorem}
\newtheorem{corollary}{Corollary}
\newtheorem{proposition}{Proposition}

\theorembodyfont{\rmfamily}
\newtheorem{remark}{Remark}

\newenvironment{proof}{\begin{ProofwCaption}{Proof}}{\end{ProofwCaption}}
\newenvironment{proof*}[1]{\begin{ProofwCaption}{{#1}}}{\end{ProofwCaption}}
\newenvironment{ProofwCaption}[1]%
  {\addvspace\theorempreskipamount \noindent{\it #1.}\rm}%
  {\qed \par \addvspace\theorempostskipamount}
\newcommand{\qedsymbol}{{\rm $\Box$}}
\newcommand{\qed}{\hfill\qedsymbol}
\newcommand{\CC}{{\mathbb C}}

\newcommand{\ZZ}{{\mathbb Z}}

\newcommand{\calB}{{\mathcal B}}
\newcommand{\calX}{{\cal X}}

\newcommand{\calD}{{\mathcal D}}

\newcommand{\eps}{\varepsilon}

\title{A note on distinguished bases of singularities}
\author{Wolfgang Ebeling
\thanks{
Partially supported by the DFG-programme SPP1388 ''Representation Theory''.
Mathematical Subject Classification - MSC2010: 32S30, 57R45, 58K40.
}
}

\date{}

\begin{document}
\selectlanguage{english}

\maketitle

\begin{center}
{\it Dedicated to the memory of Egbert Brieskorn with great admiration}
\end{center}

\begin{abstract} For an isolated hypersurface singularity which is neither simple nor simple elliptic, it is shown that there exists a distinguished basis of vanishing cycles which contains two basis elements with an arbitrary intersection number. This implies that the set of Coxeter-Dynkin diagrams of such a singularity is infinite, whereas it is finite for the simple and simple elliptic singularities. For the simple elliptic singularities, it is shown that the set of distinguished bases of vanishing cycles is also infinite. We also show that some of the hyperbolic unimodal singularities have Coxeter-Dynkin diagrams like the exceptional unimodal singularities.
\end{abstract}

%%%%%%%%%%%%%%%%%%%%%%%%
\section*{Introduction}
Let $f:(\CC^n,0) \to (\CC,0)$ be the germ of an analytic function with an isolated singularity at the origin defining a singularity $(X_0,0)$ where $X_0=f^{-1}(0)$.  We assume $n \equiv 3 \, ({\rm mod}\, 4)$. (This  can be achieved by a stabilization of the singularity.) An important invariant of the singularity $(X_0,0)$ is the symmetric bilinear intersection form $\langle \cdot , \cdot \rangle$ on the Milnor lattice of $f$. It is well known that this intersection form is negative definite if and only if the singularity is simple, it is negative semidefinite if and only if the singularity is simple elliptic and it is indefinite otherwise. It follows from the Cauchy-Schwarz inequality that for negative definite symmetric bilinear forms, the only values of $\langle x,y \rangle$ for non-collinear vectors $x,y$ with $\langle x,x \rangle= \langle y,y \rangle = -2$ are $0, \pm 1$. For semidefinite symmetric bilinear forms, also the values $\pm 2$ can be achieved. Another invariant of the singularity is the class of distinguished bases of vanishing cycles $\calB^\ast$ of the singularity \cite{BrArcata}. The intersection matrix with respect to a distinguished basis of vanishing cycles is encoded by a graph, a Coxeter-Dynkin diagram of the singularity. In this way, the class $\calB^\ast$ gives rise to the class $\calD^\ast$ of Coxeter-Dynkin diagrams with respect to distinguished bases of vanishing cycles. Here we consider questions related to the finiteness of these sets. N.~A'Campo \cite{A'C} has shown that if $f$ has corank 2, then there exists a distinguished basis of vanishing cycles for $f$ such that the mutual intersection numbers are only $0$ or $\pm 1$. Here we are interested in the opposite question: For which singularities do there exist distinguished bases of vanishing cycles such that there are pairs of basis elements with an arbitrary intersection number? We show that such bases exist for all isolated hypersurface singularities which are neither simple nor simple elliptic.  This implies that for these singularities the sets $\calB^\ast$ and $\calD^\ast$ are infinite. For the simple singularities, both sets are finite. We show that for the simple elliptic singularities the set $\calD^\ast$ is finite, but not the set $\calB^\ast$. In this way, we obtain a characterisation of simple and simple elliptic singularities. For the proof, we show among other things that the hyperbolic singularities $T_{3,3,4}$, $T_{2,4,5}$, and $T_{2,3,7}$ have Coxeter-Dynkin diagrams like the exceptional unimodal singularities.

\section{Distinguished bases of vanishing cycles} \label{sect:1}
We briefly recall the definition of distinguished bases of vanishing cycles. 

Let $f_{\lambda}: U \to \CC$ be a {\em morsification} of $f$. This is a perturbation of (a representative of) $f$ (i.e., $f_0=f$) defined in a suitable neighbourhood $U$ of the origin in $\CC^n$, depending on a parameter $\lambda \in \CC$ and such that, for $\lambda \neq 0$ small enough, the function $f_\lambda$ has only non-degenerate critical points with distinct critical values. The number of these critical points is equal to the Milnor number $\mu$ of the germ $f$. Choose a small closed disc $\Delta \subset \CC$ which contains all the $\mu$ critical values of $f_\lambda$ in its interior. Let $\calX:=f_\lambda^{-1}(\Delta) \cap B_\eps$ and $X_t:= f_\lambda^{-1}(t) \cap B_\eps$ for $t \in \Delta$ where $B_\eps$ is the ball of radius $\eps$ around the origin in $\CC^n$. Assume that $\eps$ and $\lambda \neq 0$ are chosen so small that all the critical points of the function $f_\lambda$ are contained in the interior of $\calX$ and the fibre $f_\lambda^{-1}(t)$ for $t \in \Delta$ intersects the ball $B_\eps$ transversely. Choose a basepoint $s \in \CC$ on the boundary of this disc. Join the critical values to the base point $s$ by a system of non-self-intersecting paths $\gamma_1, \ldots, \gamma_\mu$ in $\Delta$ meeting only at $s$ and numbered in the order in which they arrive at $s$, where we count clockwise from the boundary of the disc (see, e.g., \cite[Figure 5.3]{Eb}). Each path $\gamma_i$ gives a vanishing cycle $\delta_i \in H_{n-1}(X_s;\ZZ)$ determined up to orientation. It satisfies $\langle \delta_i, \delta_i \rangle=-2$. After choosing orientations we obtain a system $(\delta_1, \ldots , \delta_\mu)$ of vanishing cycles which is in fact a basis of $H_{n-1}(X_s;\ZZ)$. Such a basis is called a {\em distinguished basis of vanishing cycles}.
Let $\calB^\ast$ be the set of distinguished bases of the singularity $(X_0,0)$. There is an action of the braid group $Z_\mu$ in $\mu$ strings on the set $\calB^\ast$. The standard generator $\alpha_i$, $i=1, \ldots , \mu-1$, acts as follows. Let $(\delta_1, \ldots , \delta_\mu)$ be a distinguished basis of vanishing cycles and let $s_{\delta}$ denote the reflection corresponding to $\delta$ defined by
$s_{\delta}(x) = x + \langle x,\delta \rangle \delta$. The action of $\alpha_i$ is given by
\[
\alpha_i:  (\delta_1, \ldots, \delta_\mu) \mapsto (\delta'_1, \ldots, \delta'_\mu)=(\delta_1, \ldots, \delta_{i-1}, s_{\delta_i}(\delta_{i+1}), \delta_i, \delta_{i+2}, \ldots , \delta_\mu). 
\]
We have $\langle \delta'_k, \delta'_j \rangle = \langle \delta_k, \delta_j \rangle$ for all $1 \leq k,j \leq \mu$, $k,j \neq i,i+1$, and 
\begin{eqnarray*}
 \langle \delta'_i, \delta'_{i+1} \rangle  & = & - \langle \delta_i, \delta_{i+1} \rangle, \\
 \langle \delta'_k, \delta'_{i+1} \rangle & =  & \langle \delta_k, \delta_i \rangle, \\
 \langle \delta'_k, \delta'_i \rangle & = & \langle \delta_k, \delta_{i+1}\rangle + \langle \delta_i, \delta_{i+1} \rangle \langle \delta_k, \delta_i \rangle,
\end{eqnarray*}
where $k \neq i, i+1$.
The inverse operation $\alpha_i^{-1}$ is also denoted by $\beta_{i+1}$ and is given by
\[
\beta_{i+1}: (\delta_1, \ldots, \delta_\mu) \mapsto (\delta_1, \ldots, \delta_{i-1}, \delta_{i+1}, s_{\delta_{i+1}}(\delta_i), \delta_{i+2}, \ldots , \delta_\mu).
\]
Moreover, we can change the orientation of a cycle. Let $H_\mu$ be the direct product of $\mu$ cyclic groups of order two with generators $\gamma_1, \ldots , \gamma_\mu$, where $\gamma_i$ acts on $\calB^\ast$ by
\[
\gamma_i: (\delta_1, \ldots, \delta_i , \ldots, \delta_\mu) \mapsto (\delta_1, \ldots, -\delta_i, \ldots , \delta_\mu).
\]
The braid group $Z_\mu$ acts on $H_\mu$ by permutation of the generators: $\alpha_i\gamma_i=\gamma_{i+1}$. Let $G_\mu=Z_\mu \cdot H_\mu$ be the semi-direct product. 
It is well known that the action of the group $G_\mu$ on $\calB^\ast$ is transitive (see, e.g., \cite{AGV2, Eb}). Each distinguished basis of vanishing cycles $(\delta_1, \ldots , \delta_\mu)$ determines a Coxeter-Dynkin diagram: This is a graph with a numbering of its vertices where the vertices correspond to the basis elements and the vertices corresponding to $\delta_i$ and $\delta_j$ are connected by $|\langle \delta_i, \delta_j \rangle|$ edges which are dashed if $\langle \delta_i, \delta_j \rangle <0$.
Let $\calD^\ast$ be the set of Coxeter-Dynkin diagrams corresponding to distinguished bases of vanishing cycles. The group $G_\mu$ also acts on $\calD^\ast$. See \cite{BrArcata} for the fact that the sets $\calB^\ast$ and $\calD^\ast$ determine each other.

\begin{proposition} \label{prop:1}
If there is a distinguished basis in $\calB^\ast$ which contains three vanishing cycles $\delta^{(0)}_1,\delta^{(0)}_2, \delta^{(0)}_3$ with an intersection matrix
\begin{equation} \label{eq:0}
(\langle \delta^{(0)}_i, \delta^{(0)}_j \rangle) = \left(
\begin{array}{ccc} -2 & -2 & 0 \\-2 & -2 & 1 \\0 & 1 & -2 \end{array} \right),
\end{equation}
then there is also a basis in $\calB^\ast$ which contains a pair of vanishing cycles with an arbitrary intersection number.
\end{proposition}

\begin{proof} Let $(\delta^{(0)}_1,\delta^{(0)}_2, \delta^{(0)}_3)$ be a system of vanishing cycles with an intersection matrix (\ref{eq:0}). For $k=1,2,\dots$, we define
\[ (\delta_1^{(k)}, \delta_2^{(k)}, \delta_3^{(k)}):= \beta_2^k((\delta_1^{(0)}, \delta_2^{(0)}, \delta_3^{(0)})).\]
We claim that the intersection matrix $(\langle \delta_i^{(k)}, \delta_j^{(k)} \rangle)$, $k=0,1,\ldots$,  is equal to
\[ \left( \begin{array}{ccc} -2 & -2 & (-1)^{\frac{k}{2}}k \\
-2 & -2 & (-1)^{\frac{k}{2}}(k+1)\\
(-1)^{\frac{k}{2}}k & (-1)^{\frac{k}{2}}(k+1) & -2 
\end{array}  \right) \mbox{ for } k \mbox{ even}
\]
and
\[
\quad \left( \begin{array}{ccc} -2 & 2 & (-1)^{\frac{k-1}{2}}k \\
2 & -2 & (-1)^{\frac{k+1}{2}}(k+1) \\
(-1)^{\frac{k-1}{2}}k & (-1)^{\frac{k+1}{2}}(k+1) & -2 
\end{array} \right) \mbox{ for } k \mbox{ odd}.
\]

We prove this statement by induction on $k$. For $k=0$, the statement agrees with Equation~(\ref{eq:0}). Let $k$ be even. The proof for $k$ odd is analogous and will be omitted.
We have $(\delta_1^{(k+1)}, \delta_2^{(k+1)}, \delta_3^{(k+1)})=(\delta_2^{(k)},\delta_1^{(k)}-2\delta_2^{(k)}, \delta_3^{(k)})$ and therefore
\begin{eqnarray*}
\langle \delta_1^{(k+1)}, \delta_2^{(k+1)} \rangle & = & \langle \delta_2^{(k)},\delta_1^{(k)} \rangle -2 \langle \delta_2^{(k)}, \delta_2^{(k)} \rangle =-2-2(-2)=2, \\
\langle \delta_1^{(k+1)}, \delta_3^{(k+1)} \rangle & = & \langle \delta_2^{(k)},\delta_3^{(k)} \rangle = (-1)^{\frac{k}{2}}(k+1), \\
\langle \delta_2^{(k+1)}, \delta_3^{(k+1)} \rangle & = & \langle \delta_1^{(k)},\delta_3^{(k)} \rangle -2 \langle \delta_2^{(k)}, \delta_3^{(k)} \rangle \\
& = & (-1)^{\frac{k}{2}}k -2(-1)^{\frac{k}{2}}(k+1) = (-1)^{\frac{k}{2}+1}(k+2).
\end{eqnarray*}
\end{proof}

\begin{remark} Note that the vectors
\[
f_1 := \delta^{(0)}_1-\delta^{(0)}_2,\quad
f_2 :=  \delta^{(0)}_1-\delta^{(0)}_2-\delta^{(0)}_3
\]
are isotropic (i.e., $\langle f_1,f_1 \rangle = \langle f_2, f_2 \rangle=0$) and they generate a unimodular hyperbolic plane (i.e., $\langle f_1,f_2 \rangle = \langle f_2,f_1 \rangle=1$).
\end{remark}

\begin{proposition} \label{prop:2}
If there is a distinguished basis of vanishing cycles $(\delta_1, \ldots , \delta_\mu)$ containing two vanishing cycles $\delta_1$ and $\delta_2$ with $\langle \delta_j, \delta_1 \rangle=\langle \delta_j, \delta_2 \rangle$ for all $1 \leq j \leq \mu$, then there are infinitely many distinguished bases of vanishing cycles with the same Coxeter-Dynkin diagram.
\end{proposition}

\begin{proof}  Let $(\delta_1^{(k)}, \ldots ,\delta_\mu^{(k)}):= \beta_2^k(\delta_1, \ldots , \delta_\mu)$, $k=0,1, \ldots$. Similarly to the proof of Proposition~\ref{prop:1}, one can show that
\begin{eqnarray*} 
 \delta_1^{(k)}  & = &  \left\{ \begin{array}{cl} (-1)^{\frac{k}{2}}(k\delta_2-(k-1)\delta_1) & \mbox{for } k \mbox{ even,} \\
(-1)^{\frac{k-1}{2}}(k\delta_2-(k-1)\delta_1) & \mbox{for } k \mbox{ odd,}
\end{array} \right. \\
\delta_2^{(k)} & = & \left\{ \begin{array}{cl} (-1)^{\frac{k}{2}}((k+1)\delta_2-k\delta_1) & \mbox{for } k \mbox{ even,} \\
(-1)^{\frac{k+1}{2}}((k+1)\delta_2-k\delta_1) & \mbox{for } k \mbox{ odd,}
\end{array} \right.
\end{eqnarray*}
and hence $\langle \delta_j^{(k)}, \delta_1^{(k)} \rangle=(-1)^k \langle \delta_j^{(k)}, \delta_2^{(k)} \rangle$ for all $1 \leq j \leq \mu$.
Therefore the basis $(\delta_1^{(k)}, \ldots , \delta_\mu^{(k)})$ has the same Coxeter-Dynkin diagram as the basis $(\delta_1, \ldots , \delta_\mu)$ for each even $k$.
\end{proof}

%%%%%%%%%%%%%%%%%%%%%%%%%%%%%%
\section{Coxeter-Dynkin diagrams with multiple edges}
In this section we shall prove the main result of this note:
\begin{theorem} \label{thm:main}
Let $(X_0,0)$ be an isolated hypersurface singularity which is neither simple nor simple elliptic. For an arbitrary integer $m$, there exists a distinguished basis of vanishing cycles containing a pair of basis elements with the intersection number $m$.
\end{theorem}

The sets $\calB^\ast$ and $\calD^\ast$ are finite for the simple singularities, see \cite{BrArcata}. The simple elliptic singularities are the singularities $\widetilde{E}_6$, $\widetilde{E}_7$, and $\widetilde{E}_8$. By \cite[Fig.~10]{Gab}, the simple elliptic and the hyperbolic $T_{p,q,r}$ singularities have distinguished bases of vanishing cycles with Coxeter-Dynkin diagrams of the form of Fig.~\ref{FigGabTpqr}. 
\begin{figure}
$$
\xymatrix{ 
 & &  & & *{\bullet} \ar@{==}[d] \ar@{-}[dr]  \ar@{-}[ldd] \ar@{}^{5}[r]
 & &  & &  \\
 *{\bullet} \ar@{-}[r] \ar@{}_{p+3}[d]  & {\cdots} \ar@{-}[r]  & *{\bullet} \ar@{-}[r] \ar@{}_{6}[d] & *{\bullet} \ar@{-}[r] \ar@{-}[ur]   \ar@{}_{1}[d] & *{\bullet} \ar@{-}[dl] \ar@{-}[r] \ar@{}^{4}[d] & *{\bullet} \ar@{-}[r]  \ar@{}^{2}[d] & *{\bullet} \ar@{-}[r] \ar@{}^{p+4}[d] & {\cdots} \ar@{-}[r]  &*{\bullet} \ar@{}^{p+q+1}[d]   \\
 & & &   *{\bullet} \ar@{-}[dl] \ar@{}_{3}[r]   & & & & & \\
 & &   *{\bullet} \ar@{-}[dl] \ar@{}_{p+q+2}[r]  & & & & & &  \\
& {\cdots} \ar@{-}[dl] & & & & & & &  \\
*{\bullet}  \ar@{}_{p+q+r-1}[r] & & & & & &
  } 
$$
\caption{The graph of \cite[Fig.~10]{Gab}} \label{FigGabTpqr}
\end{figure}
The following transformations 
\begin{multline} \label{eq:Tpqr}
\beta_4, \beta_3, \beta_2; \beta_6, \beta_5, \beta_4, \beta_3; \ldots ; \beta_{p+3}, \beta_{p+2}, \beta_{p+1}, \beta_p;\\
\beta_{p+4}, \beta_{p+3} \beta_{p+2}; \ldots ; \beta_{p+q+1}, \beta_{p+q}, \beta_{p+q-1};  \\
\beta_{p+q+2}, \beta_{p+q+1}; \ldots ; \beta_{p+q+r-1}, \beta_{p+q+r-2}; \gamma_p, \gamma_{p+q-1}, \gamma_{p+q+r-2}
\end{multline}
change the ordering of this graph to the ordering of the graph indicated in Fig.~\ref{FigTpqr}. We shall work with this graph. We denote the graph corresponding to the singularity $T_{p,q,r}$ by the same symbol. The simple elliptic singularities $\widetilde{E}_6$, $\widetilde{E}_7$, and $\widetilde{E}_8$ have Coxeter-Dynkin diagrams of the form $T_{3,3,3}$,  $T_{2,4,4}$, and $T_{2,3,6}$ respectively. 
\begin{figure}
$$
\xymatrix{ 
 & &  & & *{\bullet} \ar@{==}[d] \ar@{-}[dr]  \ar@{-}[ldd] \ar@{}^{p+q+r-1}[rr]
 & &  & &  \\
 *{\bullet} \ar@{-}[r] \ar@{}_{p+1}[d]  & {\cdots} \ar@{-}[r]  & *{\bullet} \ar@{-}[r] \ar@{}_{p+q-2}[d] & *{\bullet} \ar@{-}[r] \ar@{-}[ur]   \ar@{}_{p+q-1}[d] & *{\bullet} \ar@{-}[dl] \ar@{-}[r] \ar@{}^{1}[d] & *{\bullet} \ar@{-}[r]  \ar@{}^{p}[d] & *{\bullet} \ar@{-}[r] \ar@{}^{p-1}[d] & {\cdots} \ar@{-}[r]  &*{\bullet} \ar@{}^{2}[d]   \\
 & & &   *{\bullet} \ar@{-}[dl] \ar@{}_{p+q+r-2}[r]   & & & & & \\
 & &   *{\bullet} \ar@{-}[dl] \ar@{}_{p+q+r-3}[r]  & & & & & &  \\
& {\cdots} \ar@{-}[dl] & & & & & & &  \\
*{\bullet}  \ar@{}_{p+q}[r] & & & & & &
  } 
$$
\caption{The graph $T_{p,q,r}$} \label{FigTpqr}
\end{figure}
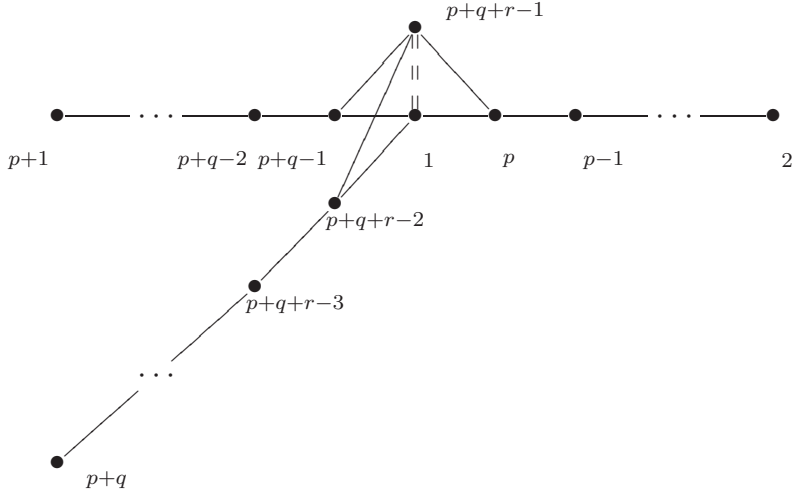
In this case, the set $\calD^\ast$ is still finite, since the intersection numbers of two vanishing cycles can only take the values $0, \pm 1, \pm 2$ (by the Cauchy-Schwarz inequality, see the introduction). On the other hand, 
the transformations
\[ \beta_5, \beta_4; \beta_4, \beta_3; \beta_3, \beta_2 \]
applied to the graph of Fig.~\ref{FigGabTpqr} effect the permutation 
\[ (1,2,3,4,5) \mapsto (3,4,5,1,2) \]
of the first 5 basis elements. This produces a distinguished basis of vanishing cycles which satisfies the conditions of Proposition~\ref{prop:2}.
Therefore, it follows from Proposition~\ref{prop:2} that the set $\calB^\ast$ is infinite. (This was already proved by P.~Kluitmann. It is stated in his thesis \cite[p.~5]{Kl} and follows from the results of the thesis, but no explicit proof is given.)  It follows from Theorem~\ref{thm:main} that both sets are infinite for the remaining singularities. Therefore we obtain the following corollary:
 
\begin{corollary} \label{cor:main}
A singularity is simple if and only if the sets $\calB^\ast$ and $\calD^\ast$ are finite. It is simple elliptic if and only if the set $\calD^\ast$ is finite, but the set $\calB^\ast$ is infinite.
\end{corollary}

For the proof of Theorem~\ref{thm:main}, we distinguish between two cases. First we consider a singularity which is neither simple nor simple elliptic nor hyperbolic. Here we use the fact that each such singularity $(X_0,0)$ deforms into an exceptional unimodal  singularity. This follows from Arnold's classification theorems. An exceptional unimodal singularity has a distinguished basis of vanishing cycles 
$(\delta_1, \ldots , \delta_{p+q+r})$
with a Coxeter-Dynkin diagram of the form of Fig.~\ref{FigEbSpqr} (see \cite[Abb.~15]{E1}). 
\begin{figure}
$$
\xymatrix{ 
 & & & & *{\bullet} \ar@{-}[d] \ar@{}^{1}[r]  &  & & & \\
 & &  & & *{\bullet} \ar@{==}[d] \ar@{-}[dr]  \ar@{-}[ldd] \ar@{}^{5}[r]
 & &  & &  \\
 *{\bullet} \ar@{-}[r] \ar@{}_{p+4}[d]  & {\cdots} \ar@{-}[r]  & *{\bullet} \ar@{-}[r] \ar@{}_{7}[d] & *{\bullet} \ar@{-}[r] \ar@{-}[ur]   \ar@{}_{2}[d] & *{\bullet} \ar@{-}[dl] \ar@{-}[r] \ar@{}^{6}[d] & *{\bullet} \ar@{-}[r]  \ar@{}^{3}[d] & *{\bullet} \ar@{-}[r] \ar@{}^{p+5}[d] & {\cdots} \ar@{-}[r]  &*{\bullet} \ar@{}^{p+q+2}[d]   \\
 & & &   *{\bullet} \ar@{-}[dl] \ar@{}_{4}[r]   & & & & & \\
 & &   *{\bullet} \ar@{-}[dl] \ar@{}_{p+q+3}[r]  & & & & & &  \\
& {\cdots} \ar@{-}[dl] & & & & & & &  \\
*{\bullet}  \ar@{}_{p+q+r}[r] & & & & & &
  } 
$$
\caption{The graph of \cite[Abb.~15]{E1}} \label{FigEbSpqr}
\end{figure}
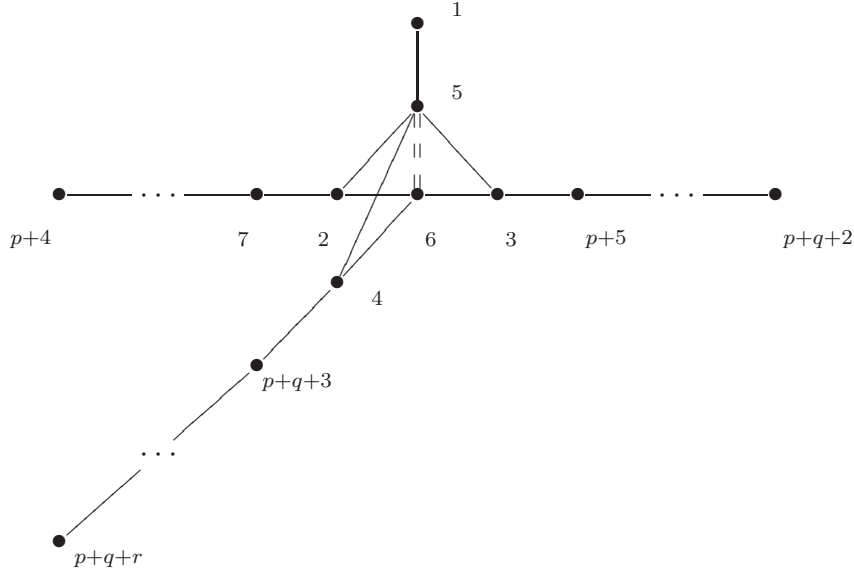
We apply the transformations
\[ \beta_6; \beta_6, \beta_5, \beta_4, \beta_3, \beta_2; \beta_6, \beta_5, \beta_4, \beta_3, \beta_2; \gamma_1, \gamma_3
\] 
to the graph of Fig.~\ref{FigEbSpqr} to get the graph $S_{p.q.r}$ of Fig.~\ref{FigSpqr}.
This basis can be extended to a distinguished basis of vanishing cycles of the singularity $(X_0,0)$ (see, e.g., \cite[Sect.~5.9]{Eb}). The distinguished basis of vanishing cycles corresponding to the graph $S_{p.q.r}$ satisfies the conditions of Proposition~\ref{prop:1}. Therefore, in this case Theorem~\ref{thm:main} follows from Proposition~\ref{prop:1}.

\begin{figure}
$$
\xymatrix{ 
 & & & & *{\bullet} \ar@{-}[d] \ar@{}^{3}[r]  &  & & & \\
 & &  & & *{\bullet} \ar@{==}[d] \ar@{-}[dr]  \ar@{-}[ldd] \ar@{}^{2}[r]
 & &  & &  \\
 *{\bullet} \ar@{-}[r] \ar@{}_{p+4}[d]  & {\cdots} \ar@{-}[r]  & *{\bullet} \ar@{-}[r] \ar@{}_{7}[d] & *{\bullet} \ar@{-}[r] \ar@{-}[ur]   \ar@{}_{4}[d] & *{\bullet} \ar@{-}[dl] \ar@{-}[r] \ar@{}^{1}[d] & *{\bullet} \ar@{-}[r]  \ar@{}^{5}[d] & *{\bullet} \ar@{-}[r] \ar@{}^{p+5}[d] & {\cdots} \ar@{-}[r]  &*{\bullet} \ar@{}^{p+q+2}[d]   \\
 & & &   *{\bullet} \ar@{-}[dl] \ar@{}_{6}[r]   & & & & & \\
 & &   *{\bullet} \ar@{-}[dl] \ar@{}_{p+q+3}[r]  & & & & & &  \\
& {\cdots} \ar@{-}[dl] & & & & & & &  \\
*{\bullet}  \ar@{}_{p+q+r}[r] & & & & & &
  } 
$$
\caption{The graph $S_{p,q,r}$} \label{FigSpqr}
\end{figure}
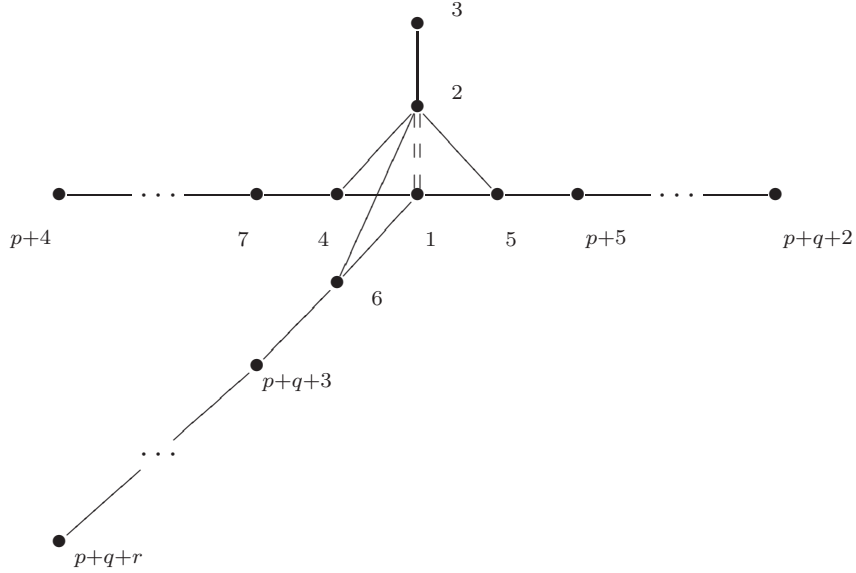

Now let $(X_0,0)$ be a hyperbolic singularity. Then each such singularity deforms to one of the singularities $T_{3,3,4}$, $T_{2,4,5}$ or $T_{2,3,7}$ (see \cite{Br}). In this case, the proof will follow from the following theorem.

\begin{theorem} \label{thm:T-S}
The singularities $T_{3,3,4}$, $T_{2,4,5}$, and $T_{2,3,7}$ have a distinguished basis of vanishing cycles with a Coxeter-Dynkin diagram of type $S_{3,3,3}$, $S_{2,4,4}$ and $S_{2,3,6}$ respectively.
\end{theorem}

\begin{proof}
The cases $T_{3,3,4}$, $T_{2,4,5}$, and $T_{2,3,7}$ are characterized among the $T_{p,q,r}$ singularities by the fact that after removing the vertex corresponding to the last basis element $\delta_{p+q+r-1}$, the remaining Coxeter-Dynkin diagrams correspond to Coxeter systems of hyperbolic type, see \cite[Ch.~5, \S 4, Exercise 12]{Bou}.
We shall consider the Coxeter-Dynkin diagrams $T_{p,q,r}$ with $(p,q,r)=(4,3,3), (5,4,2), (7,3,2)$. In these cases, the vanishing cycles 
\[\delta_1,  \delta_4, \delta_5, \ldots , \delta_{p+q+r-2}\]
form a set of simple roots of the root systems $E_6$, $E_7$ and $E_8$ respectively. Let $e$ denote the highest positive root with respect to this set. Then the vectors
\[ f_1:= e+\delta_3, \quad f_2:= e+\delta_2+\delta_3
\]
generate a unimodular hyperbolic plane (see \cite{BrLeo}). The idea is now to transform the given distinguished basis of vanishing cycles in such a way that this unimodular hyperbolic plane is exchanged with the isotropic vector $\delta_{p+q+r-1}-\delta_1$.

We now indicate the necessary base change transformations case by case.

a) The case $T_{4,3,3}$. The sequence of transformations
\[ \beta_2, \beta_3, \beta_4; \beta_3; \beta_6, \beta_5; \beta_8, \beta_7, \beta_6; \beta_6, \beta_5, \beta_4, \beta_3  \]
transforms the diagram of Fig.~\ref{Fig334}(a) to the diagram of Fig.~\ref{Fig334}(b). The sequence of transformations
\[ \beta_3; \alpha_3; \alpha_8, \alpha_7, \alpha_6, \alpha_5, \alpha_4; \alpha_4
\]
transforms this diagram to the graph of Fig.~\ref{Fig334}(c).
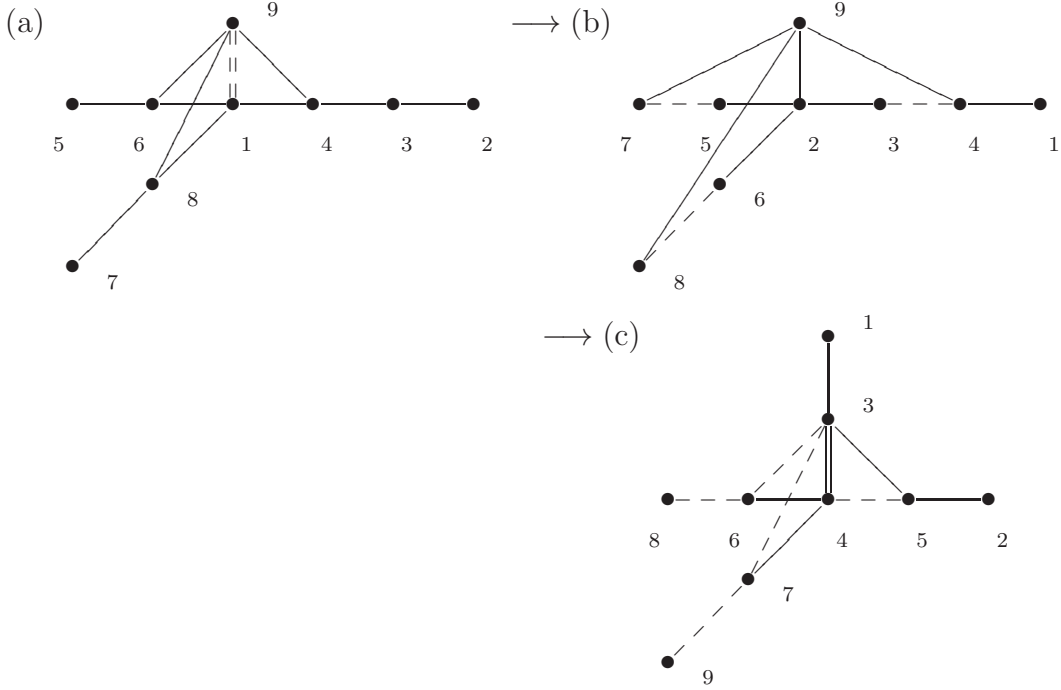
\begin{figure}
\begin{multline*}
\mbox{(a)}
\xymatrix{ 
& &  *{\bullet} \ar@{==}[d] \ar@{-}[dr]  \ar@{-}[ldd] \ar@{}^{9}[r]
 & & &  \\
*{\bullet} \ar@{-}[r] \ar@{}_{5}[d]  & *{\bullet} \ar@{-}[r]  \ar@{-}[ur] \ar@{}_{6}[d] & *{\bullet} \ar@{-}[r] \ar@{}^{1}[d] \ar@{-}[dl] & *{\bullet} \ar@{-}[r] \ar@{}^{4}[d] & *{\bullet} \ar@{-}[r] \ar@{}^{3}[d] & *{\bullet}  \ar@{}^{2}[d] \\
  & *{\bullet} \ar@{-}[dl] \ar@{}_{8}[r] & & & & \\
  *{\bullet} \ar@{}_{7}[r] & & & &&
  }
\longrightarrow \mbox{(b)}
\xymatrix{ 
& &  *{\bullet} \ar@{-}[d] \ar@{-}[dll]  \ar@{-}[drr]  \ar@{-}[dlldd] \ar@{}^{9}[r]
 & & &  \\
*{\bullet} \ar@{--}[r] \ar@{}_{7}[d]  & *{\bullet} \ar@{-}[r]  \ar@{}_{5}[d] & *{\bullet} \ar@{-}[r] \ar@{}^{2}[d] \ar@{-}[dl] & *{\bullet} \ar@{--}[r] \ar@{}^{3}[d] & *{\bullet} \ar@{-}[r] \ar@{}^{4}[d] & *{\bullet}  \ar@{}^{1}[d] \\
  & *{\bullet} \ar@{--}[dl] \ar@{}_{6}[r] & & & & \\
  *{\bullet} \ar@{}_{8}[r] & & & &&
  }\\
\longrightarrow \mbox{(c)}
\xymatrix{ 
 & &  *{\bullet} \ar@{-}[d] \ar@{}^{1}[r]  &  &  \\
 & &  *{\bullet} \ar@{=}[d] \ar@{-}[dr]  \ar@{--}[ldd] \ar@{}^{3}[r]
 & &   \\
*{\bullet} \ar@{--}[r] \ar@{}_{8}[d]  & *{\bullet} \ar@{-}[r]  \ar@{--}[ur] \ar@{}_{6}[d] & *{\bullet} \ar@{--}[r] \ar@{}^{4}[d] \ar@{-}[dl] & *{\bullet} \ar@{-}[r] \ar@{}^{5}[d] & *{\bullet}  \ar@{}^{2}[d]  \\
  & *{\bullet} \ar@{--}[dl] \ar@{}_{7}[r] & & &  \\
  *{\bullet} \ar@{}_{9}[r] & & & &
}
\end{multline*}
\caption{The case $T_{4,3,3}$} \label{Fig334}
\end{figure}
By the transformations
\[ \alpha_2,  \alpha_3, \alpha_4, \alpha_5, \alpha_6; \beta_3, \alpha_1, \alpha_2; \gamma_3, \gamma_4, \gamma_8, \gamma_9 \]
one obtains the graph $S_{3,3,3}$ with the correct ordering.
 
b) The case $T_{5,4,2}$. The sequence of transformations
\[
\beta_2, \beta_3, \beta_4, \beta_5; \beta_4; \beta_8,\beta_7,\beta_6; \beta_9, \beta_8, \beta_7; \beta_7, \beta_6, \beta_5, \beta_4
\]
transforms the diagram of Fig.~\ref{Fig245}(a) to the graph Fig.~\ref{Fig245}(b). The sequence of transformations
\[
\beta_5; \beta_9, \beta_8, \beta_7; \beta_7,\beta_6,\beta_5, \beta_4; \beta_6, \beta_5; \beta_8, \beta_7, \beta_6; \beta_6, \beta_5, \beta_4; \alpha_9, \alpha_8, \ldots , \alpha_3; \alpha_3
\]
transforms this diagram to the diagram of Fig.~\ref{Fig245}(c).
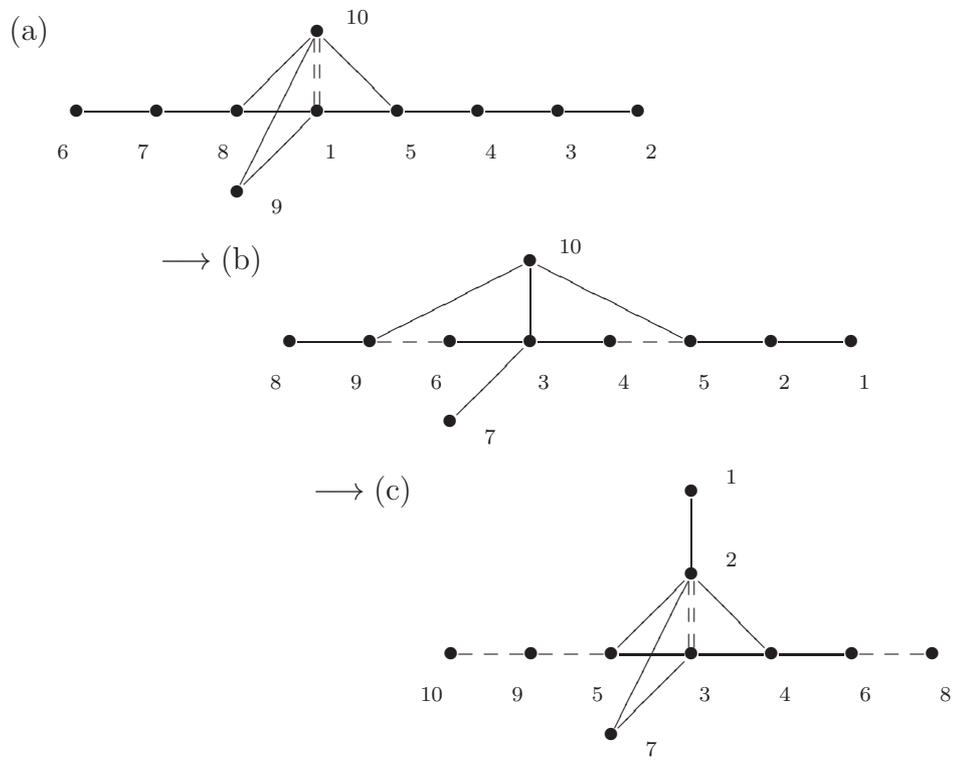
\begin{figure}
\begin{multline*}
\mbox{(a)}
\xymatrix{ 
& & &  *{\bullet} \ar@{==}[d] \ar@{-}[dr]  \ar@{-}[ldd]  \ar@{}^{10}[r]
 & & &  & \\
*{\bullet} \ar@{-}[r] \ar@{}_{6}[d] & *{\bullet} \ar@{-}[r] \ar@{}_{7}[d]  & *{\bullet} \ar@{-}[r]   \ar@{-}[ur] \ar@{}_{8}[d] & *{\bullet} \ar@{-}[r] \ar@{}^{1}[d] \ar@{-}[dl] & *{\bullet} \ar@{-}[r] \ar@{}^{5}[d] & *{\bullet} \ar@{-}[r] \ar@{}^{4}[d] & *{\bullet} \ar@{-}[r] \ar@{}^{3}[d] & *{\bullet}  \ar@{}^{2}[d] \\
  & & *{\bullet} \ar@{}_{9}[r] & & & & &   }\\
\longrightarrow \mbox{(b)}
\xymatrix{ 
& & &  *{\bullet} \ar@{-}[d] \ar@{-}[drr] \ar@{-}[dll] \ar@{}^{10}[r]
 & & &  & \\
*{\bullet} \ar@{-}[r] \ar@{}_{8}[d] & *{\bullet} \ar@{--}[r] \ar@{}_{9}[d]  & *{\bullet} \ar@{-}[r] \ar@{}_{6}[d] & *{\bullet} \ar@{-}[r] \ar@{}^{3}[d] \ar@{-}[dl] & *{\bullet} \ar@{--}[r] \ar@{}^{4}[d] & *{\bullet} \ar@{-}[r] \ar@{}^{5}[d] & *{\bullet} \ar@{-}[r] \ar@{}^{2}[d] & *{\bullet}  \ar@{}^{1}[d] \\
  & & *{\bullet} \ar@{}_{7}[r] & & & & &   }\\
\longrightarrow \mbox{(c)}
\xymatrix{ 
 & & &  *{\bullet} \ar@{-}[d] \ar@{}^{1}[r]  &  & &  \\
 & & &  *{\bullet} \ar@{==}[d] \ar@{-}[dr]  \ar@{-}[ldd] \ar@{}^{2}[r]
 & &  &  \\
*{\bullet} \ar@{--}[r] \ar@{}_{10}[d] & *{\bullet} \ar@{--}[r] \ar@{}_{9}[d] & *{\bullet} \ar@{-}[r]  \ar@{-}[ur] \ar@{}_{5}[d] & *{\bullet} \ar@{-}[r] \ar@{}^{3}[d] \ar@{-}[dl] & *{\bullet} \ar@{-}[r] \ar@{}^{4}[d] & *{\bullet} \ar@{--}[r] \ar@{}^{6}[d] & *{\bullet}  \ar@{}^{8}[d]  \\
 & & *{\bullet} \ar@{}_{7}[r] & & & & & }
\end{multline*}
\caption{The case $T_{5,4,2}$} \label{Fig245}
\end{figure}
The sequence of transformations
\[
\alpha_6, \alpha_5, \alpha_4; \beta_3, \alpha_1, \alpha_2; \gamma_2, \gamma_8, \gamma_9
\]
transforms the graph of Fig.~\ref{Fig245}(c) to the graph $S_{2,4,4}$ with the correct ordering. 

c) The case $T_{7,3,2}$. The sequence of transformations
\[
\beta_2, \beta_3, \beta_4, \beta_5, \beta_6, \beta_7; \beta_6; \beta_9,\beta_8;  \beta_{10}, \beta_9; \beta_9, \beta_8, \beta_7, \beta_6
\]
transforms the diagram of Fig.~\ref{Fig237}(a) to the graph Fig.~\ref{Fig237}(b). The sequence of transformations
\begin{multline*}
\beta_{10}, \beta_9; \beta_7, \beta_6, \beta_5; \beta_7, \beta_6, \beta_5; \beta_9, \beta_8; \beta_{10}, \beta_9; \beta_9, \ldots , \beta_5; \beta_9, \ldots , \beta_5; \\
\beta_8, \ldots , \beta_4; \beta_8, \ldots , \beta_4; \beta_8, \ldots , \beta_4; \beta_{10}; \beta_{10}, \beta_9; \beta_9, \ldots, \beta_4; \beta_{10}; \beta_{10}, \ldots , \beta_4; \\
\beta_8; \beta_8, \beta_7; \beta_7, \beta_6; \beta_6, \beta_5; \beta_9, \beta_8, \beta_7, \beta_6; \beta_6, \beta_5, \beta_4; \alpha_{10}, \alpha_9, \ldots , \alpha_3; \alpha_3
\end{multline*}
transforms this graph to the graph of Fig.~\ref{Fig237}(c).
\begin{figure}
\begin{multline*}
\mbox{(a)}
\xymatrix{ 
 & &  *{\bullet} \ar@{==}[d] \ar@{-}[dr]  \ar@{-}[ldd]  \ar@{}^{11}[r]
 & & &  & & & \\
*{\bullet} \ar@{-}[r] \ar@{}^{8}[d]  & *{\bullet} \ar@{-}[r]  \ar@{-}[ur] \ar@{}^{9}[d] & *{\bullet} \ar@{-}[r] \ar@{}^{1}[d] \ar@{-}[dl] & *{\bullet} \ar@{-}[r] \ar@{}^{7}[d] & *{\bullet} \ar@{-}[r] \ar@{}^{6}[d] & *{\bullet} \ar@{-}[r] \ar@{}^{5}[d] &*{\bullet} \ar@{-}[r] \ar@{}^{4}[d] & *{\bullet} \ar@{-}[r] \ar@{}^{3}[d] & *{\bullet}  \ar@{}^{2}[d] \\
  & *{\bullet} \ar@{}_{10}[r] & & & & & & &
   }\\
\longrightarrow \mbox{(b)}
\xymatrix{ 
 & &  *{\bullet} \ar@{-}[d] \ar@{-}[drr] \ar@{-}[dll]  \ar@{}^{11}[r]
 & & &  & & & \\
*{\bullet} \ar@{--}[r] \ar@{}^{10}[d]  & *{\bullet} \ar@{-}[r] \ar@{}^{8}[d] & *{\bullet} \ar@{-}[r] \ar@{}^{5}[d] \ar@{-}[dl] & *{\bullet} \ar@{--}[r] \ar@{}^{6}[d] & *{\bullet} \ar@{-}[r] \ar@{}^{7}[d] & *{\bullet} \ar@{-}[r] \ar@{}^{4}[d] &*{\bullet} \ar@{-}[r] \ar@{}^{3}[d] & *{\bullet} \ar@{-}[r] \ar@{}^{2}[d] & *{\bullet}  \ar@{}^{1}[d] \\
  & *{\bullet} \ar@{}_{9}[r] & & & & & & &
   }\\
\longrightarrow \mbox{(c)}
\xymatrix{ 
 & &   *{\bullet} \ar@{-}[d] \ar@{}^{1}[r]  &  & &  & &   \\
& &  *{\bullet} \ar@{==}[d] \ar@{-}[dr]  \ar@{-}[ldd]  \ar@{}^{2}[r]
 & & &  & &  \\
*{\bullet} \ar@{--}[r] \ar@{}^{11}[d]  & *{\bullet} \ar@{-}[r]  \ar@{-}[ur] \ar@{}^{5}[d] & *{\bullet} \ar@{-}[r] \ar@{}^{3}[d] \ar@{-}[dl] & *{\bullet} \ar@{-}[r] \ar@{}^{4}[d] & *{\bullet} \ar@{--}[r] \ar@{}^{6}[d] & *{\bullet} \ar@{--}[r] \ar@{}^{8}[d] &*{\bullet} \ar@{--}[r] \ar@{}^{9}[d] & *{\bullet}  \ar@{}^{10}[d] \\
  & *{\bullet} \ar@{}_{7}[r] & & & & & & 
   }
\end{multline*}
\caption{The case $T_{7,3,2}$} \label{Fig237}
\end{figure}
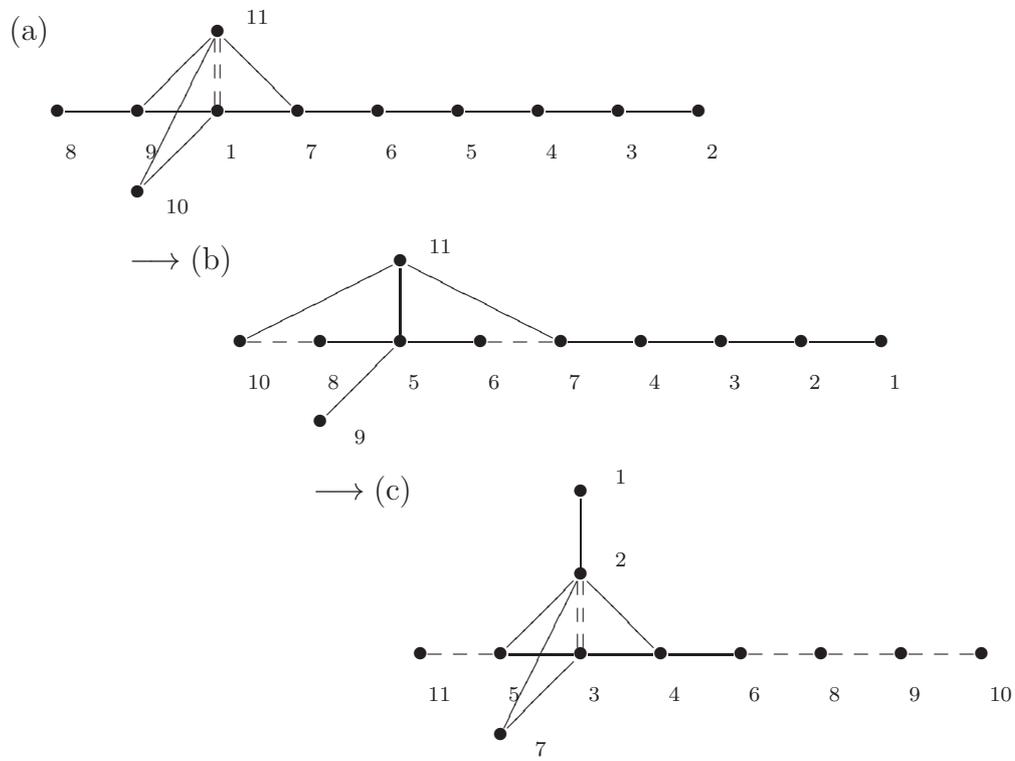
The sequence of transformations
\[
\alpha_6, \alpha_5, \alpha_4; \alpha_5 ; \alpha_{10}, \alpha_9, \alpha_8, \alpha_7; \beta_3, \alpha_1, \alpha_2; \gamma_2, \gamma_7, \gamma_9, \gamma_{11}
\]
transforms the graph of Fig.~\ref{Fig237}(c) to the graph $S_{2,3,6}$ with the correct ordering.
\end{proof}

\begin{remark} The graphs of Fig.~\ref{Fig334}(b), Fig.~\ref{Fig245}(b), and Fig.~\ref{Fig237}(b) are minimal in the sense of \cite{Ebeling96, Il}: A {\em monotone cycle} in a Coxeter-Dynkin diagram is a sequence of vertices $(\delta_{i_1}, \ldots , \delta_{i_k})$ where $i_1 <  i_2 <  \ldots  < i_k$ and $\delta_{i_j}$ is connected to $\delta_{i_{j+1}}$ for $j=1, \ldots , k$ and $j+1$ taken modulo $k$. After the changes of orientations
\[ \mbox{Fig.~\ref{Fig334}(b)}: \gamma_2, \gamma_3, \gamma_5, \gamma_6; \quad \mbox{Fig.~\ref{Fig245}(b)}: \gamma_3, \gamma_4, \gamma_6, \gamma_7; \quad \mbox{Fig.~\ref{Fig237}(b)}: \gamma_5, \gamma_6, \gamma_8, \gamma_9; 
\]
each graph contains exactly one edge of negative weight and after removing this edge it does not contain any monotone cycles.
\end{remark}

\begin{remark} Theorem~\ref{thm:T-S} shows that the graphs $S_{2,3,6}$, $S_{2,4,4}$, and $S_{3,3,3}$ of \cite[Table~3.4.2]{EbSLN} can also be realized as Coxeter-Dynkin diagrams of singularities.
\end{remark}

%%%%%%%%%%%%%%%%%%%%%%
\section*{Acknowledgements}
This note originated from a question of C.~Hertling. I am grateful to him for useful discussions and helpful comments. I would also like to thank the referee for carefully reading the paper and for useful comments which helped to improve the paper.

%%%%%%%%%%%%%%%%%%%%%%%%

\bigskip
\noindent Leibniz Universit\"{a}t Hannover, Institut f\"{u}r Algebraische Geometrie,\\
Postfach 6009, D-30060 Hannover, Germany \\
E-mail: ebeling@math.uni-hannover.de\\

\end{document}